\documentclass[runningheads,a4paper,11pt]{article}
\usepackage{amssymb,amsmath,amsthm}
\usepackage{subfigure}

\def\Z{\mathbb Z}
\def\R{\mathbb R}
\def\N{\mathbb N}

\def\A{\mathcal A}
\def\B{\mathcal B}

\newcount\sometextcount
\def\sometext#1 {{\sometextcount0
  \loop\ifnum\sometextcount<#1  \advance\sometextcount1
  Some text
  \repeat}}

\newtheorem{theorem}{Theorem}
\newtheorem{lemma}[theorem]{Lemma}
\newtheorem{remark}[theorem]{Remark}
\newtheorem{proposition}[theorem]{Proposition}
\newtheorem{example}[theorem]{Example}

\begin{document}
\title{\bf OPTIMAL NUMBER REPRESENTATIONS IN NEGATIVE BASE
}
\author{Z. MAS\'AKOV\'A and E. PELANTOV\'A\\[1mm]
{\normalsize Department of Mathematics FNSPE, Czech Technical University in Prague}\\
{\normalsize Trojanova 13, 120 00 Praha 2, Czech Republic}\\
{\normalsize emails: zuzana.masakova@fjfi.cvut.cz, edita.pelantova@fjfi.cvut.cz}}

\maketitle




\begin{abstract}{For a given base $\gamma$ and a digit set $\B$ we consider optimal representations of a number $x$, as defined by Dajani at al.\ in 2012. For a non-integer negative base $\gamma=-\beta<-1$ and the digit set $\A_\beta:=\{0,1,\dots,\lceil\beta\rceil-1\}$ we derive the transformation which generates the optimal representation, if it exists. We show that -- unlike the case of negative integer base -- almost no $x$ has an optimal representation. For a positive base $\gamma=\beta>1$ and the alphabet $\A_\beta$ we provide an alternative proof of statements obtained by Dajani et al.}
\end{abstract}


\section{Introduction}

We consider positional numerations systems given by a base $\gamma\in\R$, $|\gamma|>1$ , and a finite set (called alphabet) $\A\subset\R$,
whose elements are called digits. An expression of a real number $x$ of the form
\begin{equation}\label{eq:rep}
x=\frac{b_1}{\gamma}+\frac{b_2}{\gamma^2}+\frac{b_3}{\gamma^3}+\cdots\qquad\text{where $b_i\in\A$ for $i=1,2,3,\dots$}
\end{equation}
is called $(\gamma,\A)$-representation of $x$. The set of all $x$ having a $(\gamma,\A)$-representation is denoted by $J_{\gamma,\A}$. For a base $\gamma=\beta>1$ one obviously has $J_{\beta,\A}\subseteq[\frac{a_0}{\beta-1},\frac{a_m}{\beta-1}]$,
where $a_0<a_1<\cdots<a_m$ are the digits of $\A$. Pediccini~\cite{Pediccini} derived a necessary and sufficient condition for the alphabet $\A$ so that $J_{\beta,\A}=[\frac{a_0}{\beta-1},\frac{a_m}{\beta-1}]$.

The study of positional numeration systems with a positive non-iteger base $\beta>1$ and the alphabet $\A=\{0,1,\dots,\lfloor\beta\rfloor\}$ started in 1957 R\'enyi~\cite{Renyi}. He gave an algorithm providing to every $x\in[0,1)$ the so-called greedy expansion $x=\sum_{i=1}^\infty\frac{b_i}{\beta^i}$ using the greedy $\beta$-transformation $T_G:[0,1)\to[0,1)$,
\begin{equation}\label{eq:TG}
T(x)=\beta x - \lfloor\beta x\rfloor\,.
\end{equation}
The greedy representation of a number $x$ is then defined as $d(x)=b_1b_2b_3\cdots$, where $b_i=\big\lfloor\beta T_G^{i-1}(x)\big\rfloor$.

Besides the greedy expansion, a real number $x$ may have more  $(\beta,\A)$-representations. In fact, as shown by Sidorov~\cite{Sidorov}, almost every $x$ has continuum of representations.
The greedy expansion of $x$ can be characterized in two ways:
\begin{itemize}
\item[(a)] The sequence $b_1b_2b_3\cdots$ is lexicographically the greatest among all $(\beta,\A)$-representations of $x$.
\item[(b)] We have $0\leq x-\sum_{i=1}^n\frac{b_i}{\beta^i}<\frac{1}{\beta^n}$ for every $n=1,2,3,\dots$.
\end{itemize}

In~\cite{DDKL} the authors introduce the notion of optimal representations. A $(\beta,\A)$-representation $\sum_{i=1}^\infty\frac{c_i}{\beta^i}$ of a real number $x$ is called optimal, if for every other $(\beta,\A)$-representation $\sum_{i=1}^\infty\frac{b_i}{\beta^i}$ of $x$ one has
\begin{equation}\label{eq:1}
\Big|x-\sum_{i=1}^n\frac{c_i}{\beta^i}\Big| \leq \Big|x-\sum_{i=1}^n\frac{b_i}{\beta^i}\Big|\qquad \text{for all $n=1,2,3,\dots$.}
\end{equation}
The main part of~\cite{DDKL} is devoted to the study of existence of optimal representations for a base $\beta>1$ and the alphabet $\A=\{0,1,2,\dots,\lfloor\beta\rfloor\}$. Property (b) of the greedy expansion of $x$ ensures that this is the only candidate for the optimal
representation of $x$. The authors of~\cite{DDKL} show that among all positive bases, exceptional role is played by the so-called confluent numbers. Recall that $\beta>1$ is said confluent, if it is a zero of the polynomial
\begin{equation}\label{eq:confluent}
x^{d+1}-mx^{d}-md^{d-1}-\cdots-mx-p-1\,, \qquad\text{where \ $m,p\in\N$, \ $0\leq p<m$.}
\end{equation}
Such a polynomial is irreducible and its zero $\beta>1$ is a Pisot number, i.e., an algebraic integer with conjugates in modulus smaller than 1. Confluent Pisot numbers and related numeration systems have other exceptional properties, see~\cite{FrougnyConfluent},~\cite{BernatConfluent},~\cite{Complexity},~\cite{Edson}.

The result of Theorem 1.3 in~\cite{DDKL} states that when $\beta$ is cofluent, then every $x\in J_{\beta,\A}$ has an optimal representation.
If $\beta$ is not confluent, then the set of numbers $x\in J_{\beta,\A}$ with an optimal representation is nowhere dense and has Lebesgue measure zero. The authors also study optimal representations for negative integer base. Unlike the case of positive integer base systems, where
any $x$ has optimal representation, for bases $\gamma=-\beta\in\{-2,-3,-4,\dots\}$ and alphabet $\A=\{0,1,\dots,\beta-1\}$, optimal representation exists only for numbers $x$ with unique representation.

In our paper we focus on systems with negative non-integer base $\gamma=-\beta<-1$, $\beta\notin\Z$, and the alphabet $\A=\{0,1,\dots,\lfloor\beta\rfloor\}$. We show that almost every  $x\in J_{-\beta,\A}$ has no optimal representation (see Theorem~\ref{t:hlavni}). For positive non-integer base $\beta>1$, and $\A=\{0,1,\dots,\lfloor\beta\rfloor\}$ we give an alternative simpler proof of Theorem~1.3 from~\cite{DDKL}.

\section{Optimal transformation}

As we have explained, for a positive base $\beta>1$ and the alphabet $\A=\{0,1,\dots,\lceil\beta\rceil-1\}$, the only candidate for the optimal representation of $x\in[0,1)$ is the greedy expansion of $x$ generated by the R\'enyi $\beta$-transformation $T_G$.
Let us derive for a general real base $\gamma$, $|\gamma|>1$, and an alphabet $\B\subset\R$, which transformation must generate the first digit of the optimal representation $x=\frac{c_1}{\gamma^1}+\frac{c_2}{\gamma^2}+\frac{c_3}{\gamma^3}+\cdots$ of an $x\in J_{\gamma,\B}$ (if it exists). Since $\gamma x-c_1 = \frac{c_2}{\gamma}+\frac{c_3}{\gamma^2}+\cdots\in J_{\gamma,\B}$ and inequality~\eqref{eq:1} must be satisfied for $n=1$, we assign the first digit $D(x):=c_1$ so that the following conditions hold,
\begin{enumerate}
\item[(i)] $\gamma x - D(x) \in J_{\gamma,\B}$;
\item[(ii)] $|\gamma x- D(x)| \leq |\gamma x - b|$ for every $b\in\B$ such that $\gamma x -b\in J_{\gamma,\B}$.
\end{enumerate}

Let us mention that for some $x$ conditions (i) and (ii) can be satisfied simultaneously by two different digits. The set of such numbers $x$ -- let us denote it by $E$ -- is finite and for our purposes it is not important which value of $D(x)$ is chosen. For simplicity, we specify the digit assigning function $D:J_{\gamma,\B} \to \B$ so that it is right continuous. With this in hand, we may define the transformation $T_o:J_{\gamma,\B} \to J_{\gamma,\B}$ by the prescription 
$$
T_o(x)=\gamma x-D(x)\,.
$$  
Recall that the mapping $T_o$ is defined to ensure validity
of~\eqref{eq:1} for $n=1$, which motivates us to call it the optimal transformation. However, requiring the validity of~\eqref{eq:1} for all $n\geq 1$ gives the following theorem.

\begin{theorem}\label{t:1}
Let $\gamma\in\R$, $|\gamma|>1$, and let $\B\subset\R$ be finite. Let $\sum_{i=1}^\infty\frac{b_i}{\gamma^i}$ be the optimal $(\gamma,\B)$-representation of a number $x\in J_{\gamma,\B}$ . Then either $x\in \bigcup_{i=0}^\infty T_o^k(E)$ or $b_k=D\big(T^{k-1}(x)\big)$ for every $k\in\{1,2,3,\dots\}$.
\end{theorem}

\begin{remark}~

\begin{enumerate}
\item Since $E$ is a finite set, the union $\bigcup_{i=0}^\infty T_o^k(E)$ is at most countable.
\item If $\gamma>1$ and $\B\subset[0,+\infty)$, then $J_{\gamma,\B}\subset[0,+\infty)$. Condition (i), i.e., $\gamma x - D(x)\in J_{\gamma,\B}$,
then implies $\gamma x -D(x) \geq 0$, and requirement (ii) thus can be read without absolute values. Therefore $D(x)$ is uniquely determined for every $x\in J_{\gamma,\B}$ and in this case $E=\emptyset$.
\item If moreover $\B=\{0,1,\dots,\lceil\beta\rceil-1\}$, the transformation $T_o$ coincides with the greedy transformation $T_G$ from~\eqref{eq:TG}.
\end{enumerate}
\end{remark}

\section{Optimal transformation for negative bases}

In the whole section we consider a base $\gamma=-\beta$ with $\beta>1$, $\beta\notin \Z$, and the alphabet $\A=\{0,1,\dots,\lfloor\beta\rfloor\}$. It can be readily seen that any $x\in \Big[-\frac{\beta\lfloor\beta\rfloor}{\beta^2-1},\frac{\lfloor\beta\rfloor}{\beta^2-1}\Big]$ has at least one $(-\beta,\A)$-representation. For simplifying the notation, we write $J=J_{-\beta,A}$ and $l=-\frac{\beta\lfloor\beta\rfloor}{\beta^2-1}$,
$r=\frac{\lfloor\beta\rfloor}{\beta^2-1}$, i.e., $J=[l,r]$.
For the description of the optimal transformation $T_o$ in this case, we shall first study condition (i).

\begin{lemma}\label{claim1}
Let $a,b\in\A=\{0,1,\dots,\lfloor\beta\rfloor\}$, $a<b$, and let $x\in J$ be such that $-\beta x -a \in J$ and $-\beta x -b \in J$. Then $b=a+1$ and $-\beta x \in [l+a+1,r+a]$.
\end{lemma}

\begin{proof}
By assumption, we have
\begin{equation}\label{eq:3}
l\leq -\beta x -b < -\beta x - a \leq r\,.
\end{equation}
Therefore $r-l\geq b-a \geq 1$. If $\lfloor\beta\rfloor = 1$, then $\A=\{0,1\}$ and clearly $a=0$, $b=1$. If $\lfloor\beta\rfloor \geq 2$, then $r-l=\frac{\lfloor\beta\rfloor}{\beta-1}<2$, which implies $b=a+1$. Substituting $b=a+1$ into~\eqref{eq:3} we obtain the statement.
\end{proof}

Lemma~\ref{claim1} shows that when deciding about the assignment of the digit $D(x)$, condition (i) allows at most two possibilities, namely $a$ and $a+1$ for some $a\in\A$. Moreover, the choice is unique, unless $-\beta x\in [l+a+1,r+a]$. If this happens, by condition (ii) priority is given to the digit $a$, if $\big|-\beta x -a\big| < \big|-\beta x - (a+1)\big|$, which can be equivalently written
\begin{equation}\label{eq:4}
-\beta x < a+\frac12\,.
\end{equation}
Assignment of the digit $D(x)$ thus depends on the fact whether $a+\frac12$ belongs to the interval $[l+a+1,r+a]$ or not. Since $l+a+1<a+\frac12$ for any $\beta>1$ and $a\in\A$, we have
$$
a+\frac12\in [l+a+1,r+a] \quad\iff\quad a+\frac12 \leq r+a \quad\iff\quad r\geq \frac12\,.
$$
One can easily compute that
\begin{equation}\label{eq:5}
\frac12\leq r = \frac{\lfloor\beta\rfloor}{\beta^2-1} \quad\iff\quad \beta\in (1,\sqrt3]\cup (2,\sqrt5]\,.
\end{equation}

In case that $r<\frac12$, the condition~\eqref{eq:4} is valid for every $x$ such that $-\beta x \in [l+a+1,r+a]$, and thus if (i) allows two digits $a$ and $a+1$, by (ii) we chose for $D(x)$ the smaller one.

If $r\geq \frac12$ and $x$ is such that (i) allows two digits $a$ and $a+1$, then $D(x)=a$ if $-\beta x \in (l+a+1,a+\frac12)$ and  $D(x)=a+1$ if $-\beta x \in (a+\frac12,r+a)$. If $-\beta x=a+\frac12$, we chose $D(x)=a$, because of our convention that $T_o$ is right continuous.

\begin{proposition}\label{p}
Let $\beta>1$, $\beta\notin\N$, $\beta\neq (1,\sqrt3]\cup(2,\sqrt5]$, and $\A=\{0,1,\dots,\lfloor\beta\rfloor\}$. The optimal transformation
$T_o$ is of the form
$$
T_o=
\begin{cases}
-\beta x  & \text{for } x\in -\frac1\beta J = [-\frac{r}{\beta},r]\,,\\
-\beta x -a & \text{for } x\in -\frac{a}{\beta} + [-\frac{r}{\beta},-\frac{r}{\beta}+\frac1\beta)\,.
\end{cases}
$$
Let $\beta\in (1,\sqrt3]\cup(2,\sqrt5]$, and $\A=\{0,1,\dots,\lfloor\beta\rfloor\}$. The optimal transformation
$T_o$ is of the form
\begin{equation}\label{eq:2}
T_o=
\begin{cases}
-\beta x  & \text{for } x\in [-\frac{1}{2\beta},r]\,,\\
-\beta x -\lfloor\beta\rfloor & \text{for } x\in [l,-\frac{\lfloor\beta\rfloor-1}{\beta}-\frac1{2\beta})\,,\\
-\beta x -1 & \text{otherwise.}
\end{cases}
\end{equation}
\end{proposition}

Note that in~\eqref{eq:2}, the prescription splits into three cases if $\beta\in(2,\sqrt5]$, i.e., $\lfloor\beta\rfloor=2$, and only two cases
if $\beta\in(1,\sqrt3]$, i.e., $\lfloor\beta\rfloor=1$.

\begin{figure}
\begin{center}
\subfigure[Base $-\frac12(1+\sqrt5)$.]{
\setlength{\unitlength}{0.75pt}
\begin{picture}(175,170)
\put(5,5){\line(1,0){160}}
\put(5,165){\line(1,0){160}}
\put(5,5){\line(0,1){160}}
\put(165,5){\line(0,1){160}}
\put(105,5){\line(0,1){160}}
\put(5,105){\line(1,0){160}}
\put(5,165){\line(2,-3){69}}
\put(165,5){\line(-3,5){90}}
%
%
\put(75,59){\circle{5}}
\put(5,165){\circle*{5}}
\put(75,154){\circle*{5}}
\put(165,5){\circle*{5}}
\put(75,102){\line(0,1){6}}
\put(102,154){\line(1,0){6}}
\put(102,57){\line(1,0){6}}
\put(-3,105){\makebox(0,0){$-1$}}
\put(172,105){\makebox(0,0){$\frac1\tau$}}
\put(114,154){\makebox(0,0){$\frac12$}}
\put(118,57){\makebox(0,0){$-\frac12$}}
\put(78,95){\makebox(0,0){$-\frac1{2\tau}$}}
\put(40,95){\makebox(0,0){$-\frac1{\tau}$}}
\end{picture}
\label{fig:subfig1}
}
\qquad\qquad
\subfigure[Base $-\frac12(3+\sqrt5)$.]{
\setlength{\unitlength}{0.74pt}
\begin{picture}(175,170)
\put(5,5){\line(1,0){161}}
\put(5,166){\line(1,0){161}}
\put(5,5){\line(0,1){161}}
\put(166,5){\line(0,1){161}}
\put(120,5){\line(0,1){161}}
\put(5,120){\line(1,0){161}}
\put(5,166){\line(2,-5){48}}
\put(53.25,166){\line(2,-5){48}}
\put(166,5){\line(-2,5){64}}
%
%
\put(53.5,44){\circle{5}}
\put(101.5,44){\circle{5}}
\put(5,166){\circle*{5}}
\put(53.25,166){\circle*{5}}
\put(101.5,166){\circle*{5}}
\put(166,5){\circle*{5}}
%
\put(117,44){\line(1,0){6}}
%
\put(-1,120){\makebox(0,0){$l$}}
\put(173,120){\makebox(0,0){$r$}}
\put(137,38){\makebox(0,0){$r-1$}}
\end{picture}
\label{fig:subfig2}
}
\end{center}
\caption {Optimal transformation (a) for the base $-\beta$, where $\beta=\tau$, the golden mean, and (b)
 for the base $-\beta$, where $\beta=\tau^2$.}
\label{f1}
\end{figure}
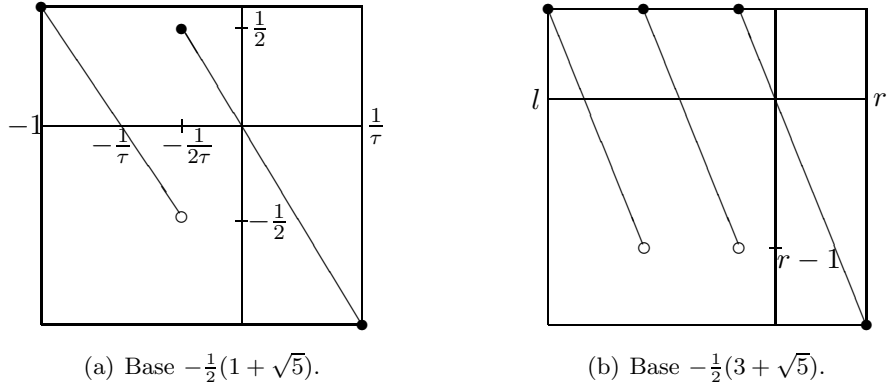

\begin{example}
Let $\beta=\tau=\frac12(1+\sqrt5)$. Then $r=\frac1{\tau^2-1}=\frac1\tau$, $l=-1$. Since $\beta$ belongs to $(1,\sqrt3]$, the optimal transformation in this case has the prescription
$$
T_o(x) = \begin{cases}
-\tau x, &\text{ for }x\in[-\frac1{2\tau}, \frac1\tau]\,,\\
-\tau x -1, &\text{ for }x\in[-1,-\frac1{2\tau})\,.
\end{cases}
$$
The transformation is depicted in Figure~\ref{f1} (a).
\end{example}

\begin{example}
Let $\beta=\tau^2=\frac12(3+\sqrt5)$. Then $r=\frac2{\tau^4-1}$, $l=-\frac{2\tau^2}{\tau^4-1}$. Since $\beta$ does not belong to $(1,\sqrt3]\cup(2,\sqrt5]$, the optimal transformation in this case has the prescription
$$
T_o(x) = \begin{cases}
-\tau^2 x, &\text{ for }x\in[-\frac{r}{\tau^2}, r]\,,\\
-\tau^2 x -1, &\text{ for }x\in[-\frac{r}{\tau^2}-\frac1{\tau^2},-\frac{r}{\tau^2})\,,\\
-\tau^2 x -2, &\text{ for }x\in[l,-\frac{r}{\tau^2}-\frac1{\tau^2})\,,\\
\end{cases}
$$
The transformation is depicted in Figure~\ref{f1} (b).
\end{example}

Note that for the case $\beta>1$, $\beta\notin\N$, $\beta\neq (1,\sqrt3]\cup(2,\sqrt5]$, we have for every discontinuity point $\delta$ of $T_o$ that
$$
\lim_{x\to\delta+} T_o(x) = r \quad\text{and}\quad \lim_{x\to\delta-} T_o(x) = r -1\,.
$$
By the result of Li and Yorke~\cite{LiYorke} for the transformation $T_o$ there exists a unique absolutely continuous invariant measure (acim) for $T_o$, hence $T_o$ is ergodic.  It is easy to see that the support of the acim is the full interval $[l,r]$.

For the case $\beta\in (1,\sqrt3]\cup(2,\sqrt5]$, the transformation $T_o$ has either one (if $\lfloor\beta\rfloor=1$) or two (if $\lfloor\beta\rfloor=2$) discontinuity points. For every discontinuity point $\delta$, we have
$$
\lim_{x\to\delta+} T_o(x) = \frac12 \quad\text{and}\quad \lim_{x\to\delta-} T_o(x) = -\frac12\,.
$$
Again, by Li and Yorke, there exists a unique absolutely continuous $T_o$-invariant ergodic measure. Its support however may not be the full interval, but contains all discontinuity points of $T_o$ in its interior.

\section*{Optimal representations for negative bases}

In~\cite{DDKL} the authors study systems with negative integer bases $-\beta$, $\beta\in\N$, $\beta\geq 2$, and the alphabet $\A=\{0,1,\dots,\beta-1\}$. In such a system, every $x\in J_{-\beta,\A}$ has at most two representations. It is shown that if $x$ has
two representations, than none of them is optimal. However, there are only countably many elements with more than one representation. 
This means that almost every $x\in J_{-\beta,\A}$ has a unique, and thus also optimal representation. The aim of this section is to show that for negative non-integer bases, the situation is different.

\begin{theorem}\label{t:hlavni}
Let $\beta>1$, $\beta\notin\N$, and let $\A=\{0,1,\dots,\lfloor\beta\rfloor\}$. Then almost every $x\in J_{-\beta,\A}$ has no optimal representation.
\end{theorem}

\begin{proof}
Consider the optimal transformation $T_o$. Recall that the set of its points of discontinuity was denoted by $E$. If $x$ does not belong to the countable union $S:=\bigcup_{k=0}^\infty T_o^k(E)$, then only the representation generated by $T_o$ can be optimal (see Theorem~\ref{t:1}). We will show that there is an interval $I\subset J_{-\beta,\A}$ such that for any $y\in I$, its representation generated by $T_o$ is not optimal. 
In determining the desired interval I, we distinguish three cases:
\begin{description}
\item{\bf Case 1: $\beta\in (1,\sqrt3]\cup(2,\sqrt5]$.} By~\eqref{eq:5}, this is equivalent to $r\geq \frac12$.  Put $I=(-\frac1{2\beta},-\frac{\{\beta\}}{2\beta^2})$, where we denote $\{\beta\}=\beta-\lfloor\beta\rfloor$.
    Using the explicit form of the transformation $T_o$ from Proposition~\ref{p}, we can verify that $I\subset(-\frac1{2\beta},r)$ and $T_o(I)\subset(-\frac1{2\beta},r)$. Therefore the representation of $x\in I$ generated by $T_o$ is of the form
    \begin{equation}\label{eq:6}
    x=\frac{0}{-\beta} + \frac{0}{(-\beta)^2} + \sum_{i=3}^\infty \frac{c_i}{(-\beta)^i}\,.
    \end{equation}
    On the other hand, the choice of the right end-point of the interval $I$ ensures that
    \begin{equation}\label{eq:7}
    |x|>\Big|x+\frac{\{\beta\}}{\beta^2}\Big| = \Big|x-\big(\frac{1}{-\beta} + \frac{\lfloor\beta\rfloor}{(-\beta)^2}\big)\Big|\,.
    \end{equation}
    It can be easily checked that
    \begin{equation}\label{eq:8}
    x-\Big(\frac{1}{-\beta} + \frac{\lfloor\beta\rfloor}{(-\beta)^2}\Big) \in \frac1{\beta^2}[l,r] = \frac1{\beta^2} J_{-\beta,\A}\,,
    \end{equation}
    i.e., there exist digits $d_3,d_4,d_5,\dots\in\A$ so that
    $x=\frac{1}{-\beta} + \frac{\lfloor\beta\rfloor}{(-\beta)^2} + \sum_{i=3}^\infty \frac{d_i}{(-\beta)^i}$. Inequality~\eqref{eq:7}
    excludes that the representation~\eqref{eq:6} of $x$ is optimal, since it contradicts~\eqref{eq:1}.
\item{\bf Case 2: $\frac12 > r> \frac{\{\beta\}}{2}$.} Put $I=(-\frac{r}{\beta},-\frac{\{\beta\}}{2\beta^2})$.
\item{\bf Case 3: $r\leq \frac{\{\beta\}}{2}$.} Put $I=(-\frac{r}{\beta},\frac{r}{\beta^2}-\frac{\{\beta\}}{\beta^2})$.
\end{description}

The argumentation for cases 2 and 3 is similar to that of Case 1. For every $x\in I$, the representation of $x$ generated by $T_o$ is of the form~\eqref{eq:6}, and $x$ satisfies~\eqref{eq:7} and~\eqref{eq:8}. Thus $x$ has no optimal representation.

We have shown that no $x\in I$ has an optimal representation.
For the proof of Theorem~\ref{t:hlavni} it suffices to show that for almost every $x\in J_{-\beta,A}$ there is a $k\in\N$ such that $T_o^k(x)\in I$. This is a consequence of the Birkhoff ergodic theorem.
In fact, in Cases 2 and 3, the interval $I$ is a subset of the support of the acim. In Case 1, this may not be the case. Nevertheless, since the interval $I=(-\frac1{2\beta},-\frac{\{\beta\}}{2\beta^2})$ has a discontinuity point $\delta=-\frac1{2\beta}$ as its end-point, the intersection of $I$ and the support of the acim is a non-degenerate interval and the statement is also established.
\end{proof}

\section*{Optimal representations for positive bases}

For a positive integer base $\gamma = \beta\in\N$ and the alphabet of digits $\A=\{0,1,\dots,\beta-1\}$, numbers in $[0,1)$ have either a unique one, or two representations of the form~\eqref{eq:rep}, and one easily shows that the greedy representation of every $x\in[0,1)$ is the optimal representation. Let us therefore focus on non-integer bases $\beta>1$ with the alphabet $\A=\{0,1,\dots,\lfloor\beta\rfloor\}$
The question of optimal representations in such systems has been completely solved in~\cite{DDKL}. Also the idea behind our proof of Theorem~\ref{t:hlavni} in the previous section is an analogue of that of~\cite{DDKL}, namely, finding an interval $I$ of positive length in which the greedy expansion of any $x$ is not an optimal representation. The existence of such an interval $I$ separates between confluent and non-confluent bases. This, together with further technical details, is the content of Propositions 2.1 and 3.1. in~\cite{DDKL}. These propositions, together with the fact that the R\'enyi greedy transformation $T_G$ on $[0,1)$ is ergodic provides a straightforward proof of Theorem 1.3 from ~\cite{DDKL}.

In what follows, we provide an alternative simpler proof of the exceptional stand of confluent bases. We will need the Parry characterization of greedy expansions, see~\cite{Parry}.

Let us recall some facts about the greedy representation $x=\sum_{i=1}^\infty\frac{x_i}{\beta^i}$ of a real number $x\in[0,1)$ in base $\beta$ written as $d(x) = x_1x_2x_3\cdots$. It can be shown that the natural ordering $<$ of reals in $[0,1)$ corresponds to the lexicographic ordering $\prec$ of their greedy representations. Formally, for every $x,y\in[0,1)$ we have $x<y$ if and only if $d(x)\prec d(y)$. For the description of digit strings  arising as greedy representations of numbers in $[0,1)$, crucial role is played by the infinite R\'enyi expansion of 1, namely the string $d^*(1) = \lim_{x\to 1-}d(x)$.
Parry~\cite{Parry} has shown that a digit string $y_1y_2y_3\cdots \in \A^\N$ is  equal to $d(y)$ for some $y\in[0,1)$ if and only if
\begin{equation}\label{eq:10}
y_iy_{i+1}y_{i+2}\cdots \prec d^*(1) \qquad \text{for every }\ i=1,2,3,\dots \,.
\end{equation}
Digits strings in $\A^\N$ with the property~\eqref{eq:10} are called admissible. It is also shown that every suffix of $d^*(1)$ is lexicographically smaller or equal to $d^*(1)$.

Taking for the base $\beta$ a confluent Pisot number, zero of the polynomial~\eqref{eq:confluent}, the infinite R\'enyi expansion of 1 is of the form $d^*(1)=(mm\cdots m p)^\omega$. Since $\beta\in(m,m+1)$, the canonical alphabet $\A$ is composed of digits $0,1,2,\dots,m=\lfloor\beta\rfloor$.
The non-admissible digit strings $y_1y_2y_3\cdots \in \A^\N$ with a finite number of non-zero digits must contain a substring $y_iy_{i+1}\cdots y_{i+n} = m^na$ with $a\in\A$, $a>p$. In~\cite{FrougnyConfluent}, Frougny studies confluent numeration systems. From her results, one can easily derive that every $y_1y_2\cdots y_k0^\omega \in \A^\N$ such that $y=\sum_{j=1}^k\frac{y_j}{\beta^j}\in[0,1)$ can be reduced to an admissible string $x_1x_2\cdots x_k0^\omega$ representing the same number $y=\sum_{j=1}^k\frac{x_j}{\beta^j}$,
using the rewriting rule $bm^na \to (b+1)0^n(a-p-1)$, $b\in\{0,1,\dots,m-1\}$, $a\in\{p+1,\dots,m\}$. In other words, we have the following proposition.

\begin{proposition}\label{p:frougny}
Let $\beta>1$ be a confluent Pisot number, i.e., let $d^*(1)=(mm\cdots mp)^\omega$ where $m,p\in\N$, $0\leq p< m$. If $\sum_{j=1}^\infty\frac{y_j}{\beta^j}$ is a representation of a number $y$ in the canonical alphabet $\A=\{0,1,\dots,m\}$ such that $y_j=0$ for all $j> k$, then the greedy representation of $y$ is of the form $y=\sum_{j=1}^\infty\frac{x_j}{\beta^j}$, where also $x_j=0$ for all $j> k$.
\end{proposition}

\begin{proposition}\label{p:A}
Let $\beta>1$ satisfy $d^*(1)=(mm\cdots mp)^\omega$ where $m,p\in\N$, $0\leq p< m$. Then every $x\in[0,1)$ has an optimal representation.
\end{proposition}

\begin{proof}
Let $ \sum_{j=1}^\infty\frac{y_j}{\beta^j}=\sum_{j=1}^\infty\frac{x_j}{\beta^j}$ be two representations of a number $x\in[0,1)$ and let $d(x)=x_1x_2x_3\cdots$. Since the greedy representation $d(x)$ is lexicographically the largest among all representations of the number $x$, we have $x_1x_2x_3\cdots \succeq y_1y_2y_3\cdots$. Necessarily, for every $k\in\N$, we have
\begin{equation}\label{eq:nerovnost}
x_1x_2\cdots x_k 0^\omega \succeq y_1y_2\cdots y_k0^\omega\,.
\end{equation}
We distinguish two cases:
\begin{itemize}
\item[\bf a)] Let $y_1\cdots y_k0^\omega$ be admissible. Then it is the greedy expansion of some number in $[0,1)$. Since the lexicographic order preserves the order of the corresponding real numbers, inequality~\eqref{eq:nerovnost} implies
    $$
    \sum_{j=1}^k\frac{y_j}{\beta^j}\leq \sum_{j=1}^k\frac{x_j}{\beta^j}\,,
    $$
    whence
    $$
    x-\sum_{j=1}^k\frac{x_j}{\beta^j} \leq x-\sum_{j=1}^k\frac{y_j}{\beta^j}\,.
    $$

\item[\bf b)] Let $y_1\cdots y_k0^\omega$ be non-admissible. By Proposition~\ref{p:frougny} there exists an admissible string $y'_1\cdots y'_k0^k$ such that
 $\sum_{j=1}^k\frac{y_j}{\beta^j}=\sum_{j=1}^k\frac{y'_j}{\beta^j}$. Thus $x$ has a representation of the form $x=\sum_{j=1}^k\frac{y'_j}{\beta^j}+\sum_{j=k+1}^\infty\frac{y_j}{\beta^j}$ for which
 $ y'_1\cdots y'_ky_{k+1}y_{k+2}\cdots \preceq x_1x_2x_3\cdots$.
 Using the same argumentation as in case a), we have
 $$
 x-\sum_{j=1}^k\frac{x_j}{\beta^j} \leq x-\sum_{j=1}^k\frac{y'_j}{\beta^j} = x-\sum_{j=1}^k\frac{y_j}{\beta^j}\,.
 $$
\end{itemize}

The discussion shows that~\eqref{eq:1} is satisfied for every $k\geq 1$, and therefore the greedy representation of $x$ is optimal.
\end{proof}

\begin{proposition}
Let $\beta>1$, $\beta\notin\Z$ and let $d^*(1)\neq (m\cdots mp)^\omega$, $m,p\in\N$, $0\leq p<m$. Then there exists an interval $I\subset[0,1]$ such that no $x\in I$ has an optimal representation.
\end{proposition}

\begin{proof} Denote $d^*(1)=t_1t_2t_3\cdots$ and put $i:=\min\{j\geq 2 \mid t_j<t_1\}$, i.e.,
$d^*(1) = t_1\cdots t_1 t_{i} t_{i+1\cdots}$. Realize that necessarily
\begin{equation}\label{eq:nonconfluent}
t_{i+1}t_{i+2}\cdots \prec d^*(1)=t_1t_2\cdots\,,
\end{equation}
since otherwise $t_{i+1}t_{i+2}\cdots = t_1t_2\cdots$, hence $d^*(1)$ is purely periodic with period of length $i$. In particular,
$d^*(1) = (t_1t_1\cdots t_1 t_{i})^\omega$, which in turn means that $\beta$ is a confluent Pisot number. Inequality~\eqref{eq:nonconfluent} implies that 
$$
\sum_{j=1}^\infty\frac{t_{j+i}}{\beta^j}<1\,.
$$
Define a number $R=1+\frac1{\beta^i}$ and 
$$
L=\frac{t_1}{\beta}+\frac{t_2}{\beta}+\cdots+\frac{t_{i-1}}{\beta^{i-1}}+\frac{t_i+1}{\beta^i}\,.
$$  
We have
$$
1<1+\frac1{\beta^i}\Big(\underbrace{1-\sum_{j=1}^\infty\frac{t_{j+i}}{\beta^j}}_{>0}\Big) = \sum_{j=1}^{\infty}\frac{t_j}{\beta^{j}} + \frac1{\beta^{i}} - \sum_{j=1}^\infty\frac{t_{j+i}}{\beta^{j+i}} = L <R\,.
$$
Find an integer $k$ such that $I:=\frac1{\beta^k}(L,R)\subset(0,1)$.
We will show that no $x\in I$ has an optimal representation.

Let $x=\sum_{j=1}^\infty\frac{x_j}{\beta^j}$ be the greedy representation of $x\in I$. Since 
$$
\frac1{\beta^k}<\frac1{\beta^k}L<x<\frac1{\beta^k}R=\frac1{\beta^k}+\frac1{\beta^{k+i}}\,,
$$ 
we have $x_1=x_2=\cdots = x_{k-1}=0$, $x_k=1$, and $x_{k+1}=\cdots=x_{k+i}=0$. Let  $\sum_{j=1}^\infty\frac{y_j}{\beta^j}$ be the greedy representation of the number $y=x-\frac1{\beta^k}L$. 
As $x<\frac1{\beta^{k}}R$ and $L>1$, we have $y<\frac1{\beta^k}(R-1)=\frac1{\beta^{k+i}}$. Then necessarily $y_1=y_2=\cdots=y_{k+i}=0$. Therefore
$$
x=\underbrace{\frac{t_1}{\beta^{k+1}}+\cdots + \frac{t_{i-1}}{\beta^{k+i-1}} + \frac{t_i+1}{\beta^{k+i}}}_{=\frac1{\beta^k}L} + \underbrace{\frac{y_{k+i+1}}{\beta^{k+i+1}}+\frac{y_{k+i+2}}{\beta^{k+i+2}}+\cdots}_{=y}
$$
is also a representation of the number $x$. However,
$$
x-\sum_{j=1}^{k+i} \frac{x_j}{\beta^j} = x-\frac1{\beta^k} > x-\frac1{\beta^k}L\,.
$$
The latter contradicts inequality~\eqref{eq:1}, i.e., the greedy representation of $x$ is not optimal. As explained before, this means that $x$ does not have an optimal representation.

\end{proof}

\section{Comments}

Similarly as in~\cite{DDKL}, we consider numeration systems with digits in the alphabet $\A=\{0,1,\dots,m\}$, where $m$ is a minimal integer such that $J_{\gamma,\A}$ is an interval. It would be interesting to study how the choice of the alphabet influences existence of an optimal representation. In particular, one can ask whether in the system with symmetric alphabet $\B:=(-\frac{\beta+1}{2},\frac{\beta+1}{2})\cap\Z$, considered by Akiyama and Scheicher~\cite{AkiyamaScheicher}, there exist exceptional bases with properties analogous to those of confluent Pisot numbers in the system with the alphabet $\{0,1,\dots,\lceil\beta\rceil-1\}$.

\section*{Acknowledgements}

We acknowledge the financial support by the Czech Science Foundation
grant GA\v CR 201/09/0584, and by the grants MSM6840770039 and LC06002
of the Ministry of Education, Youth, and Sports of the Czech
Republic.


%


\end{document}